\documentclass[10pt,reqno]{amsart}
\usepackage{amsmath}
\usepackage{amssymb}
\usepackage{amsthm}
\usepackage{eepic,epic}
\usepackage{epsfig}
\usepackage{graphicx}
\usepackage{colortbl}
\usepackage{stackengine}
\usepackage{esint}

\newlength{\defbaselineskip}
\newcommand{\setlinespacing}[1]%
           {\setlength{\baselineskip}{#1 \defbaselineskip}}

\numberwithin{equation}{section}
\def\Xint#1{\mathchoice
	{\XXint\displaystyle\textstyle{#1}}%
	{\XXint\textstyle\scriptstyle{#1}}%
	{\XXint\scriptstyle\scriptscriptstyle{#1}}%
	{\XXint\scriptscriptstyle\scriptscriptstyle{#1}}%
	\!\int}
\def\XXint#1#2#3{{\setbox0=\hbox{$#1{#2#3}{\int}$ }
		\vcenter{\hbox{$#2#3$ }}\kern-.6\wd0}}

\def\dashint{\Xint-}

\newtheorem{thm}{Theorem}[section]

\newtheorem{lem}[thm]{Lemma}

\theoremstyle{definition}
\newtheorem{defn}[thm]{Definition}

\theoremstyle{remark}
\newtheorem{rem}[thm]{Remark}
\numberwithin{equation}{section}

\begin{document}

\title[weighted Hessian estimates]
{Weighted Hessian estimates in Orlicz spaces for nondivergence elliptic operators with certain potentials}

\author{Mikyoung Lee and Yoonjung Lee}

\thanks{
M. Lee was supported by the
National Research Foundation of Korea by the Korean Government (NRF-2021R1A4A1032418).
Y. Lee was supported by the
National Research Foundation of Korea by the Ministry of Education (NRF-2022R1I1A1A01068481)}


\address{Mikyoung Lee, Department of Mathematics and Institute of Mathematical Science, Pusan national University, Busan 46241, Republic of Korea}
\email{mikyounglee@pusan.ac.kr}

\address{Yoonjung Lee, Department of Mathematics,
Yonsei University, Seoul 03722, Republic of Korea}
\email{yjglee@yonsei.ac.kr}

\begin{abstract}
We prove interior weighted Hessian estimates in Orlicz spaces for nondivergence type elliptic equations with a lower order term which involves a nonnegative potential satisfying a reverse H\"older type condition. 
\end{abstract}

\maketitle
	
\section{Introduction}

On the $L^p$ regularity theory, extensive research activities have been conducted for nondivergence elliptic equations with varying assumptions on coefficients, 
with numerous studies in for instance \cite{BBHV, CFL91, CFL93, CZ52, D12, GT, KK07, T, V, W} and references therein.
Recent studies on such regularity problem have been extending to Orlicz spaces
which were first introduced in \cite{O} as a generalization of $L^p$ spaces, see for instance \cite{BL15, BLO17, BOPS, WYZJ09,Ya16}.    
This paper investigates the weighted $L^p$ regularity problem in the setting of Orlicz spaces for nondivergence elliptic equations that include a lower order term of the form
\begin{equation}\label{second_order_ell}
-a_{ij} D_{ij} u+Vu=f \ \ \text{ in } \Omega.
\end{equation}
Here,
$\Omega$ is an open and bounded domain in $\mathbb{R}^n$ with $n\ge3$,
and the coefficient matrix $\textbf{A}=(a_{ij})$ is supposed  to be symmetric ($a_{ij}=a_{ji}$ for $i,j=1,...,n$), satisfy the uniform ellipticity condition
\[
\mu^{-1}|\xi|^2 \leq \langle \textbf{A}(x) \xi_i, \xi_j \rangle \leq \mu |\xi|^2 \qquad \text{for}  \quad \forall \xi \in \mathbb{R}^n,\   x\in \Omega 
\]
for some constant $\mu \ge 1$ and belong to
 the space of bounded mean oscillation (BMO)
functions with small BMO semi-norms (see  Definition~\ref{smallBMOcondition}).

The nonnegative potential $V$ is assumed to be in the class $ \mathcal{RH}_q$ for some $q \ge n/2$, whose definition is that $V \in L_{loc}^q$ and 
there exists a constant $c_q>0$ such that the \textit{reverse H\"older inequality} 
$$
\left( \dashint_{B}V(x)^q dx \right)^{1/q} \leq c_q\, \dashint_{B} V(x)dx
$$ 
holds for every ball $B$ in $\mathbb{R}^n$.
(We refer to \cite{S} for the definition of the class $\mathcal{RH}_{\infty}$.)
A typical example in the class $\mathcal{RH}_q$ is  $V(x)=|x|^{\alpha}$ for $\alpha>-n/q$.
Moreover, it is not difficult to check that the class $\mathcal{RH}_q$ is invariant under scaling, translation and scalar multiplication.
Indeed, it guarantees that
\begin{equation}\label{invariant}
V \in \mathcal{RH}_q \quad \Longleftrightarrow  \quad r^2 V(r\cdot+y) \in \mathcal{RH}_q
\end{equation}
for any $r>0$, $y\in \mathbb{R}^n$ and $1\leq q \leq\infty$. 
Another property of this class is that $\mathcal{RH}_{q_1} \subseteq \mathcal{RH}_{q_2}$ when $q_2 \leq q_1 $ by the H\"older inequality. 
Also, there exists $\sigma \in (0,1)$ depending on $q$ and $n$ such that
$ \mathcal{RH}_q \subseteq \mathcal{RH}_{q+\sigma}$ as the self-improving property
(see \cite[Theorem~9.3.3]{Gr09} for more in-depth details).

In the case $V=0$, the $W^{2,p}$ regularity estimates for the equation \eqref{second_order_ell} have been studied over several decades in for instance \cite{CZ52, W} for the case $a_{ij}=\delta_{ij}$, 
\cite{GT} for continuous coefficients $a_{ij} \in C(\overline{\Omega})$ and 
\cite{CFL91,CFL93} for discontinuous coefficients $a_{ij}$ belonging to VMO (vanishing mean oscillation) space.
In the next breath, many researchers intend to establish such regularity results in the setting of more generalized spaces than $L^p$ spaces. 
In particular, regularity estimates in the Orlicz space were derived in \cite{WYZJ09} for the Poisson equation when $a_{ij}=\delta_{ij}$. 
While, in \cite{BL}, weighted $L^p$ regularity was taken into account with Muckenhoupt weight $w\in A_p$ for \eqref{second_order_ell}  without a lower order term and with small BMO coefficients (see also \cite{Ya15}). 
Considering local Muckenhoupt weight, 
a little different form of weighted $W^{2, p}$ estimates was proved in \cite{L15}.  In addition,
we refer to \cite{BLO15} for exponent variable spaces $L^{p(\cdot)}$
and \cite{HO15} for generalized Orlicz spaces, also known as \textit{Musielak-Orlicz spaces}.

The $L^p$-type regularity for the equation \eqref{second_order_ell} in general case $V\in \mathcal{RH}_q$ with some $q\ge n/2$ can be traced back to the work of Shen \cite{Sh95} for $a_{ij}=\delta_{ij}$. 
More precisely,
he proved the global $L^p$ estimates
\begin{equation}\label{Lp regularity_V}
\|D^2 u\|_{L^p(\mathbb{R}^n)} +\|Vu\|_{L^p(\mathbb{R}^n)} \leq c\|f\|_{L^p(\mathbb{R}^n)}
\end{equation}
for $1<p\leq q$ and any solution $u$ of $(-\Delta+V)u=f$, 
by using estimates for the integral representation of the fundamental solution for the operator $-\Delta+V$. 
We also note that before his work
there were similar results of \eqref{Lp regularity_V} in \cite{T} for $V=|x|^2$ with $p=2$ 
and \cite{Z} for a nonnegative polynomial $V$.
When the coefficients $a_{ij}$ belong to VMO space,
Vitanza \cite{Vi93} showed  the global $W^{2,p}$ estimates for
any solution $u \in W^{2, p}\cap W_{0}^{1, p}(\Omega)$ of \eqref{second_order_ell} 
 when $V\in L^q$ with $q>n/2$ for $1<p \leq n/2$ and $q=p$ for $p>n/2$. We also refer to \cite{V} for the related work.
After then, for $V\in \mathcal{RH}_q$ with some $q \ge n/2$,
the $W^{2, p}$ regularity estimates were established in \cite{BBHV} 
with $1<p\leq q$. 
In the weighted setting, Yao \cite{Ya14} studied weighted type of \eqref{Lp regularity_V} for Schr\"odinger operator $-\Delta+V$ under $V\in \mathcal{RH}_{\infty}$ including any nonnegative polynomial $V$.
 We also refer to \cite{CVV} for some weighted type $W^{2, p}$ estimates with local Muckenhoupt weight. 
On the other hand, Orlicz-type regularity result for the nondivergence elliptic equations with a lower order term is not even known so far.

Our aim in this paper is to establish the interior Hessian estimates in the spaces $L_w^{\Phi}$ for \eqref{second_order_ell} under the assumptions that the coefficients $a_{ij}$ satisfy a small BMO condition and the potential $V \in \mathcal{RH}_q$ for some $q \ge n/2$.
Here, the weighted Orlicz space $L_{w}^{\Phi}$ is equipped with the $N$-function $\Phi$ satisfying a $\nabla_2 \cap \Delta_2$-condition and the Muckenhoupt weight $w \in i(\Phi)$. 
It is worth noting that
the weighted Orlicz space $L_{w}^{\Phi}$ with $w\equiv 1$ corresponds to the classical Orlicz space $L^{\Phi}$, which implies that our result also covers the case of Orlicz regularity without any weight. The details of these concepts are given in Section \ref{se2}.
Also, such weighted type regularity was studied in $L^{p}$ spaces for the special case $a_{ij}=\delta_{ij}$ in \cite{Ya14} under the assumption $\mathcal{RH}_{\infty}$ as mentioned before. 
We would like to mention that the condition on $V$ of our main result is more relaxed than that in \cite{Ya14}, from
the relation $\mathcal{RH}_{\infty} \subseteq \mathcal{RH}_q$ for $q \ge n/2$.

The paper is organized as follows. 
In Section \ref{se2}, we introduce some basic concepts of weighted Orlicz spaces and their properties and then state our main result, Theorem \ref{mainresult}.  
Our approach in this paper is based on the maximal function method together with modified Vitali covering lemma that was developed in \cite{C89} and improved in \cite{W}.  
In order to adapt it, it is vital to obtain decay estimates of the upper-level set for maximal function of $D^2u$.
Such decay estimates are established in Section \ref{se4}.
Section \ref{se5} is devoted to the proof of Theorem \ref{mainresult}.

\section{Preliminary and Main result}\label{se2}
Let us start with some standard notation. 
For $x\in \mathbb{R}^n$, let $B_r(x):= \{y\in \mathbb{R}^n\,:\, |x-y|<r\}$ be denoted as a ball centered at $x$ with radius $r>0$.
For simplicity of notation, we set $B_r =B_r(0)$.
Let $\Omega$ be an open and bounded domain in $\mathbb{R}^n$.
We use the notation
$$
\overline{f}_{\Omega}=\dashint_{\Omega} f(x) \, dx = \frac{1}{|\Omega|}\int_{\Omega} f(x) dx  $$ 
where $|\Omega|=\int_{\Omega} dx$ is a standard Euclidean measure of $\Omega$. 
The Sobolev space $W^{2,p}(\Omega)$ consists of functions $u \in L^{p}(\Omega)$ such that $Du, D^{2}u \in L^{p}(\Omega)$.
The space $W_0^{1, p}(\Omega)$ is defined for $1<p<\infty$ as a subset of Sobolev spaces $W^{1, p}(\Omega)$ consisting all functions vanishing zero on the boundary $\partial \Omega$.

We start with some definitions and useful properties of weighted Orlicz spaces.
For more details, we refer to the survey \cite{GP77, KK91, KR61}. 

\subsection{N-function}\label{subsec2.1}
Let us first introduce \textit{N-function} $\Phi:[0, \infty) \rightarrow [0, \infty)$ which is
convex, i.e. 
$$
\Phi(\lambda \rho) \leq \lambda\Phi(\rho)\ \text{for any } \lambda \in [0,1],
$$
continuous, increasing such that
$\Phi(0)=0$, $\Phi(\rho)>0$ for all $\rho>0$
and 
$$
\lim_{\rho \rightarrow 0^+} \frac{\Phi(\rho)}{\rho} = \lim_{\rho \rightarrow \infty} \frac{\rho}{\Phi (\rho)} = 0.
$$

In this paper, we shall additionally assume $\Delta_2 \cap \nabla_2$-condition on N-function $\Phi$ (i.e., $\Phi \in \Delta_2 \cap \nabla_2$) which means that 
\begin{enumerate}
	\item[(1)] (\textit{$\Delta_2$-condition})
	there exists a constant $\tau_1 >1$
	such that
	$$\Phi(2\rho) \leq \tau_1 \Phi(\rho) \quad \text{for all} \quad \rho>0 $$
	and
	\item[(2)] (\textit{$\nabla_2$-condition})
	there exists $\tau_2 >1$ such that 
	$$\Phi(\rho) \leq \frac{1}{2\tau_2}\Phi(\tau_{2}\rho) \quad \text{for all} \quad \rho>0. $$ 
\end{enumerate}

In particular for $\Phi \in \Delta_2$, 
there exists $\tau=\tau(\lambda)>0$ such that $\Phi(\lambda \rho) \leq \tau \Phi(\rho) $ for all $\rho, \lambda > 0$ thanks to the convexity of $\Phi$. 
We then define 
\[
h_{\Phi} (\lambda) = \sup_{\rho >0}\frac{\Phi(\lambda \rho)}{\Phi(\rho)} \quad \text{ for } \lambda>0
\]
and \textit{lower index} of $\Phi$ is denoted by  
\[
i(\Phi) = \lim_{\lambda \rightarrow 0^+} \frac{\log(h_{\Phi}(\lambda))}{\log\lambda} = \sup_{0<\lambda<1} \frac{\log(h_{\Phi}(\lambda))}{\log\lambda}
\]
which is analogue to the role of a lower degree of homogeneity of $\Phi$.

Meanwhile, for $\Phi \in \Delta_2$ there exist $q_1, q_2 \in [1, \infty)$ with $1<q_1 \leq q_2 <\infty$ such that
\begin{equation} \label{ineq_Young}
c^{-1} \min\{\lambda^{q_1}, \lambda^{q_2}\} \leq \frac{\Phi(\lambda \rho)}{\Phi(\rho)} \leq  c\max\{\lambda^{q_1}, \lambda^{q_2}\} \quad \text{for } \lambda, \rho \geq 0
\end{equation}
where $c$ is independent of $\lambda$ and $\rho$ (see \cite{GP77, FK97}).
For the values of $q_1$ satisfying \eqref{ineq_Young} with $\lambda \geq 1 $, 
we should notice that $i(\Phi)$ is the same as the supremum of those $q_1$, namely $q_1 \leq i(\Phi)$. 
In a simple case $\Phi(\rho)=\rho^q$ with $\rho>0$, we see $i(\Phi)=q$ as a note.

\subsection{Muckenhoupt weight}
A Muckenhoupt weight $w\in A_p$, for $1<p<\infty$, 
which is a locally integrable and positive function $w$ on $\mathbb{R}^n$ satisfying
$$
[w]_{p} :=\sup_{B \subset \mathbb{R}^n} \left( \dashint_B w(x) dx \right) \left( \dashint_B w(x)^{-\frac{1}{p-1}} dx \right)^{p-1} < \infty.
$$  
We sometimes write 
$$w(\Omega)=\int_{\Omega} w(x) dx.$$

The following properties for the Muckenhoupt weights $w\in A_p$ are well-known, see \cite{Gr09} for more details with their proofs.
\begin{lem}\label{property_A}
	Let $w\in A_p$ with $1<p<\infty$. 
	The following properties hold:
	\begin{itemize}
		\item[(1)] For $p_1\leq p_2$
		\begin{equation}\label{inclusion_mucken}
		A_{p_1} \subseteq A_{p_2}.
		\end{equation}
		\item[(2)]
		There exists a small constant $\tilde \varepsilon  >0$ depending on $n,p,$ and $[w]_p$ such that 
		\begin{equation}\label{self-improving}
		A_p \subset A_{p-\tilde \varepsilon}.		
		\end{equation}
		\item[(3)] For measurable sets $C, D$ with $C \subset D \subset \mathbb{R}^n$,
		there exists $\sigma  \in (0,1)$ depending on $n,p,$ and $[w]_p$ such that
		\begin{equation}\label{measure comparison}
		[w]_{p}^{-1}\left(\frac{|C|}{|D|}\right)^p\leq  \frac{w(C)}{w(D)} \leq c \left(\frac{|C|}{|D|} \right)^{\sigma}
		\end{equation}
		where a constant $c>0$ depends only on $n$ and $p$. 
	\end{itemize}
\end{lem}

\subsection{Weighted Orlicz spaces}
For given N-function $\Phi\in \Delta_2 \cap \nabla_2$ and weight $w$,
we define \textit{weighted Orlicz space} $L^{\Phi}_w(\Omega)$
as a set of all Lebesgue measurable functions $g$ on $\Omega$ such that 
$$\varrho_{\Phi,w}(g; \Omega):=  \int_{\Omega} \Phi(|g(x)|)w(x) dx < \infty.$$
Here, $\varrho_{\Phi,w}$ is called a \textit{modular of $\Phi$} in the weighted setting.
This space $L_w^\Phi(\Omega)$ is a reflexive Banach space equipped with the following \textit{Luxemburg norm}
\begin{equation}\label{luxemburg}
\|g\|_{L_w^\Phi(\Omega)} = \inf \left\{ \lambda>0 : \varrho_{\Phi,w}\bigg(\frac{g}{\lambda}, \Omega\bigg) \leq 1 \right\}.
\end{equation}
It is easy to check that $\|\lambda g\|_{L_w^{\Phi}}=|\lambda| \|g\|_{L_{w}^{\Phi}}$ for $\lambda \in \mathbb{R}$ and $\|f+g\|_{L_w^{\Phi}} \leq \|f\|_{L_w^{\Phi}}+\|g\|_{L_w^{\Phi}}$.
In fact, the Luxemburg norm \eqref{luxemburg} is specifically related to modular $\varrho_{\Phi,w}$ as
the \textit{unit ball property}
\begin{equation}\label{unitballproperty}
\varrho_{\Phi,w}(g; \Omega) \le 1 \ \ \Longleftrightarrow \ \ \Vert g \Vert_{L^{\Phi}_w(\Omega)} \le 1,
\end{equation} 
and
\[
 \|g\|_{L_w^\Phi(\Omega)} -1  \leq  \varrho_{\Phi,w}(g;\Omega)\leq c \left( \|g\|_{L_w^\Phi(\Omega)}^{q_2} +1 \right)
 \]
where the constant $c>1$ is independent of $g$, due to the property \eqref{ineq_Young}. 
	
Let us naturally define a \textit{weighted Orlicz Sobolev space} $W^{2,\Phi}_{w}(\Omega)$ which is  a collection of functions $g \in L^{\Phi}_{w}(\Omega)$ whose distributional derivatives up to order $2$ also belong to $L^{\Phi}_{w}(\Omega)$ with norm 
$$\Vert g \Vert_{W^{2,\Phi}_{w}(\Omega)} =\|g\| _{L^{\Phi}_{w}(\Omega)}+\|Dg\| _{L^{\Phi}_{w}(\Omega)}+\|D^2g\| _{L^{\Phi}_{w}(\Omega)}. $$

By the standard measure theory and basic properties of $N$-function $\Phi$, the following property can be derived for the weighted Orlicz spaces, see \cite[Lemma 4.6]{BOPS}.

\begin{lem}\label{measure theory}
	Assume that $g$ is a nonnegative measurable function in $\Omega$, and let $\eta >0 $ and $M>1$ be constants. For any $\Phi \in \Delta_2 \cap \nabla_2$ and $w \in A_p$ with some $1<p<\infty$, we have
\[
	g \in L_w^\Phi(\Omega)  \text{ if and only if } S:= \sum_{k \geq 1} \Phi(M^{k}) w\left(\left\{x \in \Omega : g(x) > \eta M^k \right\} \right) < \infty, 
	\]
	and moreover,
	\begin{equation}\label{estimate_S}
	c^{-1} S \leq \int_{\Omega} \Phi\left(|g(x)|\right)w(x)\,dx \leq c\left(w(\Omega) + S\right),
	\end{equation}
	where the constant $c>0$ depends only on $\eta, M , \Phi (1)$, $p$ and $[w]_p$.
\end{lem}

\subsection{Main Result}

Given weight $w$ and $N$-function $\Phi \in \Delta_2 \cap \nabla_2$, we assume that
 \begin{equation}\label{mainass_w}
  w \in A_{i(\Phi)}.
  \end{equation}
This assumption is the necessary and sufficient condition under which the Hardy-Littlewood maximal operator is bounded on the corresponding weighted Orlicz space, see Lemma~\ref{Mbdd in wO}. In addition, since $\Phi \in \nabla_2$ implies $i(\Phi)>1$, the basic properties of the Muckenhoupt weight that we mentioned before can be applied to the proof of our main result under assumption \eqref{mainass_w}.

We also have the following results that will guarantee
the existence of strong solutions of our main equation. 
\begin{lem}[See Lemma~2.5 in \cite{BLO17}]\label{Phi-integrability_f}
Assume that $\Phi \in \Delta_2 \cap \nabla_2 $ and $w \in A_{i(\Phi)}$ and that $f \in L_w^{\Phi}(\Omega)$. 
Then there exists 
$\tilde{p} \in (1, i(\Phi)-\tilde{\varepsilon})$ depending on $n,\Phi, i(\Phi)$ and $[w]_{i(\Phi)}$ such that  $ f\in L^{\tilde{p}}(\Omega)$  with the estimates
\[
\|f\|_{L^{\tilde{p}}(\Omega)} \le c \|f\|_{ L_w^{i(\Phi)-\tilde\varepsilon}(\Omega)}   \le c \|f\|_{L_w^{\Phi}(\Omega)} 
\]
where $\tilde{\varepsilon}, c>0$ depend on $n, \Phi, i(\Phi),$ and $[w]_{i(\Phi)}$.
\end{lem} 
\begin{proof}
By \eqref{self-improving}, there exists a small constant $\tilde \varepsilon = \tilde \varepsilon(n, i(\Phi), [w]_{i(\Phi)}) >0$ such that 
 $w\in A_{i(\Phi)}\subset  A_{i(\Phi)-\tilde \varepsilon}$. 
	It allows $w\in A_{i(\Phi)-\tilde\varepsilon/2} $ thanks to \eqref{inclusion_mucken}.
By the definition of $i(\Phi)$, we see that $L^{\Phi}_w(\Omega)$ is continuously embedded in $L^{i(\Phi)-\tilde \varepsilon/2}_w(\Omega)$. 
In fact, by \eqref{ineq_Young} we obtain
	\begin{align*}
	\|f\|_{L_w^{i(\Phi)-\tilde{\varepsilon}/2}(\Omega)}
	&= \lambda\, \bigg( \int_{\Omega} \left( \frac{|f(x)|}{\lambda}\right)^{i(\Phi)-\frac{\tilde{\varepsilon}}{2}} w(x) dx \bigg)^{1/(i(\Phi)-\tilde{\varepsilon}/2)}\\
	&\leq \lambda
	\left( w(\Omega) + \frac{c}{\Phi(1)} \int_{\Omega} \Phi \left(\frac{|f(x)|}{\lambda}\right) w(x) dx  \right)^{1/(i(\Phi)-\tilde{\varepsilon}/2)}\\
	&\leq c  \|f\|_{L_{w}^{\Phi} (\Omega)}
	\end{align*}
	by letting $\lambda=\|f\|_{L_{w}^{\Phi}(\Omega)}$,
	where $c=c(w, \Phi, i(\Phi)) >0$.

Meanwhile, by the H\"older inequality
\[  \left(\dashint_B |f(x)|^{\frac{i(\Phi)-\tilde \varepsilon/2}{i(\Phi)-\tilde \varepsilon}}\chi_{\Omega}\, dx\right)^{i(\Phi)-\tilde \varepsilon}  \leq \frac{c}{w(B)} \int_{B} \left( |f(x)|^{\frac{i(\Phi)-\tilde \varepsilon/2}{i(\Phi)-\tilde \varepsilon}}\chi_{\Omega}\right)^{i(\Phi)-\tilde \varepsilon} w(x) \;dx
\]
where $c=[w]_{i(\Phi)-\tilde{\varepsilon}}$.
Letting $\tilde p:= \frac{i(\Phi)-\tilde \varepsilon/2}{i(\Phi)-\tilde \varepsilon} $, we have 
\[
\| f\|_{L^{\tilde p}(\Omega)}  \leq \frac{c |B|^{\frac{i(\Phi)-\tilde \varepsilon}{i(\Phi)-\tilde \varepsilon/2}} }{w(B)^{\frac{1}{i(\Phi)-\tilde \varepsilon/2}}}\|f\|_{L^{i(\Phi)-\tilde \varepsilon/2}_w(\Omega)}
\]
where $B$ is a ball containing $\Omega$.

\end{proof}

Throughout the paper, the coefficient matrix $\textbf{A}=(a_{ij})$  is supposed to be symmetric and uniformly elliptic.
Our principal assumption on the coefficient matrix  $\textbf{A}$ is following:
\begin{defn}\label{smallBMOcondition}
For $\delta, R>0$,  the coefficient matrix $\textbf{A} =(a_{ij})$ is \textit{$(\delta, R)$-vanishing}
 if 
 \begin{equation}\label{delta-vanishing}
\sup_{0<r \le R} \sup_{y\in \mathbb{R}^n} \dashint_{B_r(y)}|\textbf{A}(x)-\overline{\textbf{A}}_{B_r(y)}| dx \leq \delta.
\end{equation}
\end{defn}
The $(\delta, R)$-vanishing condition on the coefficient matrix  $\textbf{A} =(a_{ij})$ in \eqref{delta-vanishing} means coefficients $a_{ij}$ are in a set of 
BMO functions with small BMO semi-norm on every $B_r(y)$ for $y\in \mathbb{R}^n$.
We note that this condition is weaker than that of vanishing mean oscillation (VMO), because $a_{ij} \in$ VMO implies that for each $\delta>0$ there exists $R>0$ such that $a_{ij}$ is $(\delta, R)$-vanishing.

For the sake of convenience, we consider the space $W_V^{2, \gamma}(\Omega)$ defined for $1<\gamma<\infty$ as a closure of $C^{\infty}_0(\Omega)$ functions equipped with 
\begin{equation*}
\|u\|_{W_V^{2, \gamma}(\Omega)} := \|u\|_{W^{2, \gamma}(\Omega)}+\|Vu\|_{L^\gamma(\Omega)}.
\end{equation*} 
Now let us state our main result.

\begin{thm}\label{mainresult}
Given any $N$-function $\Phi \in \Delta_2 \cap \nabla_2 $, let $w \in A_{i(\Phi)}$. Assume that $V\in  \mathcal{RH}_{q}$ for some $q \ge n/2$ and 
$f\in L_{w}^{\Phi}(B_6)$.
There exists a small $\delta=\delta(n,\mu,q, c_q, \Phi,$ $i(\Phi),[w]_{i(\Phi)})>0$ such that if $\textbf{A}$ is $(\delta, 6)$-vanishing,
then for any solution $u\in W_V^{2, \gamma}(B_6) $ of \eqref{second_order_ell} in $B_6$ with $1<\gamma \leq \min\{\tilde{p},q\}$, we have $D^2 u, Vu \in L_w^{\Phi} (B_1)$ with the estimate
\begin{equation*}
\|D^2u\|_{L_w^{\Phi} (B_1)}+\|Vu\|_{L_w^{\Phi}(B_1)} \leq c\,(\|f\|_{L_w^{\Phi}(B_6)}+\|u\|_{L^\gamma(B_6)})
\end{equation*} 	
with the positive constant $c$ depending on $n,\mu, q, c_q, w, \Phi, i(\Phi)$ and $[w]_{i(\Phi)}$, where $\tilde{p}$ is given in Lemma~\ref{Phi-integrability_f}.
\end{thm}

\begin{rem}
For any $1<\gamma \leq \min\{ \tilde{p}, q \} $,
we can deal with strong solutions $u \in W_V^{2, \gamma}(B_6)$ of the  main equation \eqref{second_order_ell} in $B_6$ provided that $f \in L^{\Phi}_w(B_6) \subset L^{\tilde{p}}(B_6) \subset L^{\gamma}(B_6) $ from Lemma~\ref{Phi-integrability_f}.
\end{rem}

\begin{rem}
As noted in the introduction, the class $ \mathcal{RH}_q$ possesses a self-improving property, allowing the range of $q$ to be extended from $q >n/2$ to $q \ge n/2$. Consequently, it is enough to prove our main result, Theorem~\ref{mainresult}, for potentials $V$ that belong to the class $ \mathcal{RH}_q$ for some $q > n/2$.
\end{rem}

\section{Maximal function approach}\label{se4}

\subsection{Maximal operator}
We recall some definitions and useful properties of Hardy-Littlewood maximal function.
For a locally integrable function $g$ defined in $\mathbb{R}^n$, we denote the Hardy-Littlewood maximal function of $g$ by
$$
\mathcal{M}g(x)=\sup_{r>0} \dashint_{B_r(x)} |g(y)| dy
$$
at each point $x\in \mathbb{R}^n$.
We denote
$$
\mathcal{M}_{\Omega}g(x)=\mathcal{M}(\chi_{\Omega}g)(x).
$$

Let us now state the basic properties of the Hardy-Littlewood maximal function.
From the definition of the maximal operator $\mathcal{M}$, one  can easily see that
\begin{equation*}
|f(x)|^{p}  
\leq (\mathcal{M}(|f(x)|^{\gamma}))^{\frac{p}{\gamma}}
\end{equation*}
for a.e. $x\in \Omega$.
On the other hand, the following estimates are well-known:
\begin{itemize}
	\item [(1)] (strong p-p estimate) for $1<p\leq \infty$,
	$$
	\|\mathcal{M}g\|_{L^p(\mathbb{R}^n)} \leq C_{n,p} \|g\|_{L^p(\mathbb{R}^n)},
	$$
	\item [(2)] (weak 1-1 estimate) for all $t>0$,
	\begin{equation}\label{weak1-1}
	|\{x\in \mathbb{R}^n\,:\, \mathcal{M}g \ge t\}| \leq \frac{C}{t}\|g\|_{L^1(\mathbb{R}^n)}.
	\end{equation}
\end{itemize}
Similar estimates hold in the weighted Orlicz setting when $w$ belongs to the certain Muckenhoupt class, see \cite{KT82} and \cite[Theorem 2.1.1]{KK91} for their proofs and additional details.
\begin{lem}\label{Mbdd in wO}
Given any $N$-function $\Phi \in \Delta_2 \cap \nabla_2 $, let $w \in A_{i(\Phi)}$.
Then there exists a positive constant $c=c(n,\Phi,w)$ such that
	\begin{equation*}
	\int_{\mathbb{R}^{n}}\Phi(|g|) w\,dx\leq \int_{\mathbb{R}^{n}} \Phi({\mathcal{M}}g)w\,dx \leq c \int_{\mathbb{R}^{n}}\Phi(|g|) w\,dx,
	\end{equation*}
	for all $g \in L_w^\Phi(\mathbb{R}^{n})$.
\end{lem}

\subsection{Key Lemmas}

The following lemma is crucial in proving our main theorem, Theorem \ref{mainresult}.
\begin{lem}\label{keylemma}
For any $\varepsilon>0$
there exists a small $\delta = \delta( n, \gamma, \mu, q, c_q, \varepsilon)>0$
so that
for
any solution $u \in W_V^{2, \gamma}(\Omega)$ of 
\begin{equation*}
-a_{ij} D_{ij} u+Vu=f \ \ \text{ in } \Omega \supset B_6
\end{equation*}
assuming that $\textbf{A}=(a_{ij})$ is $(\delta, 6)$-vanishing and
 $V\in \mathcal{RH}_q$ for some $q>n/2$,
 if $\|u\|_{L^{\gamma}(B_6)} \le \delta$ and
\begin{equation}\label{assump}
\{x\in \Omega\,:\, \mathcal{M}(|D^2u|^{\gamma})(x)\leq 1 \} \cap \{ x\in \Omega\,:\, \mathcal{M}(|f|^{\gamma})(x)\leq \delta^{\gamma}\} \cap B_1\neq \emptyset
\end{equation}
with $1<\gamma \le q,$
then we have
\begin{equation}\label{key} 
|\{x\in \Omega\,:\, \mathcal{M}(|D^2u|^{\gamma})(x)>N_1^{\gamma} \}\cap B_1| <\varepsilon |B_1|
\end{equation}
where a constant $N_1>1$ depends on $n, \gamma, \mu, q$ and $c_q $.
\end{lem}

\begin{proof}
	By \eqref{assump}, there exists $x_0 \in B_1$ such that
	\begin{equation}\label{con1}
	\dashint_{B_{\rho}(x_0)} 
	|D^2u(y)|^{\gamma}\, dy \leq 1 
	\end{equation}
	and
	\begin{equation*}
	\dashint_{B_{\rho}(x_0)} |f(y)|^{\gamma}\, dy \leq \delta^{\gamma} 
	\end{equation*}
	for all $\rho>0$. 
	Since $B_\frac{7}{2}\subset B_5(x_0)$,  
	we get 
	\begin{equation}\label{con_Vu}
	\dashint_{B_\frac{7}{2}} |D^2u(y)|^{\gamma}\, dy \leq \left(\frac{10}{7}\right)^{n}  \dashint_{B_5(x_0)} |D^2u(y)|^{\gamma}\, dy \leq 2^n 
	\end{equation}
	from \eqref{con1}.
	In a similar way,
	we have
	\begin{equation}\label{con_f}
	\dashint_{B_4} |f(y)|^{\gamma}\, dy \leq 2^{n}\delta^{\gamma}
	\end{equation}
	because $B_4 \subset B_5(x_0)$.
	Let us consider a solution $h \in W^{2, \gamma}(B_\frac{7}{2})$ to the problem 
	\begin{equation}\label{homogeneous case}
	\left\{
	\begin{aligned}
	-\overline{a_{ij}}_{B_\frac{7}{2}} D_{ij}h=0\qquad &\text{in} \quad B_\frac{7}{2},\\
	h=u\qquad &\text{on} \quad \partial B_\frac{7}{2}.
	\end{aligned}
	\right.
	\end{equation}
		Then we observe that $u-h \in W^{2, \gamma}(B_\frac{7}{2}) \cap W_0^{1, \gamma}(B_\frac{7}{2})$ solves
	\[
	\left\{
	\begin{aligned}
	-a_{ij} D_{ij}(u-h)&= f-Vu + a_{ij}D_{ij}h   &\text{in}& \quad B_\frac{7}{2},\\
	u-h&=0 &\text{on}& \quad \partial B_\frac{7}{2}.
	\end{aligned}
	\right.
	\]
	From \cite{CFL91}, 
	there exists a constant $c=c(n, \gamma, \mu)>0$ so that
	\begin{align}\label{Hest_u-h}
	\|D^2(u-h)\|_{L^{{\gamma}}(B_\frac{7}{2})}  &\le c\,\|f\|_{L^{\gamma}(B_\frac{7}{2})} + c\,\|Vu\|_{L^{\gamma}(B_\frac{7}{2})} +c\,\| a_{ij}\|_{L^{\infty}(B_\frac{7}{2})} \|D^2 h \|_{L^{\gamma}(B_\frac{7}{2})}\nonumber\\
	&  \le c\,\big(\|f\|_{L^{\gamma}(B_\frac{7}{2})} + \|Vu\|_{L^{\gamma}(B_\frac{7}{2})} + \|D^2 h \|_{L^{\gamma}(B_\frac{7}{2})}\big)
	\end{align}
	since $a_{ij}\in L^{\infty}$ from the symmetry and uniform ellipticity conditions.
		Now we use \cite[Theorem~13]{BBHV} to derive
	\[
	\|Vu\|_{L^{\gamma}(B_\frac{7}{2})} \le \,c ( \|f\|_{L^{\gamma}(B_4)}+\|u\|_{L^{\gamma}(B_4)}).
	\]
	In addition, from the boundary $L^\gamma$ estimates for the problem \eqref{homogeneous case} (see \cite[Chapter 9]{GT} or \cite{CFL93}), 
	we note that 
	\[
	\|D^2 h \|_{L^{\gamma}(B_\frac{7}{2})}\le  c \| u \|_{L^{\gamma}(B_\frac{7}{2})}.
	\]
Inserting these estimates into \eqref{Hest_u-h}, we arrive at
	\begin{align}\label{comparison}
	\|D^2(u-h)\|_{L^{{\gamma}}(B_\frac{7}{2})} 
	\le \,c\|f\|_{L^{\gamma}(B_4)}+c\|u\|_{L^{\gamma}(B_6)} \le c\delta 
	\end{align}
	from the assumption $\|u\|_{L^{\gamma}(B_6)} \le \delta$ and \eqref{con_f}.
	Here, $c$ is a positive constant depending on $n, \gamma, \mu, c_q$ and $q$.
	Then by \eqref{con_Vu} we have 
	\begin{align}\label{estimate_Vh}
	\|D^2h\|_{L^{\gamma}(B_\frac{7}{2})} 
	&\le  \|D^2(u- h)\|_{L^{\gamma}(B_\frac{7}{2})} +  \|D^2u\|_{L^{\gamma}(B_\frac{7}{2})} \nonumber\\
	&\le  c\delta+ c7^{n/\gamma} \le N_0
	\end{align}
	for some constant $N_0=N_0( n, \gamma, \mu, c_q, q)\ge1$.
	Meanwhile, applying the local $C^{1, 1}$ regularity for \eqref{homogeneous case} together with \eqref{estimate_Vh} we discover 
	\begin{align}\label{L infty estimate_Du}
	\|D^2h\|_{L^{\infty}(B_3)}^{\gamma} 
	&\le \frac{c}{|B_\frac{7}{2}|} \| D^2 h\|_{L^{\gamma}(B_\frac{7}{2})}^{\gamma}   \leq cN_0^{\gamma}. 
	\end{align}
	
	The weak 1-1 estimate \eqref{weak1-1} shows
	\begin{align*}
	&|\{ x\in B_1\,:\, \mathcal{M}_{B_3}(|D^2(u-h)|^\gamma)(x) > N_0^{\gamma}\}|\\
	&\le\frac{c}{N_0^{\gamma}}  \|D^2(u-h)\|_{L^{\gamma}(B_3)}^{\gamma}  \leq \frac{c}{N_0^{\gamma}}  \|D^2(u-h)\|_{L^{\gamma}(B_\frac{7}{2})}^{\gamma}   \le c N_0^{-\gamma}\delta^{\gamma}=\varepsilon |B_1|
	\end{align*}
	from \eqref{comparison}.
	Here, we choose $\delta = \delta( n, \gamma, \mu, c_q, q, \varepsilon)>0$ so that the last equality in the above holds. 
	To conclude \eqref{key}, we only need to show 
	$$
	\{x\in B_1\, :\, \mathcal{M}(|D^2u|^{\gamma})(x)>N_1^\gamma \} \subset \{x\in B_1\,:\, \mathcal{M}_{B_3}(|D^2(u-h)|^{\gamma})(x) >N_0^{\gamma} \}
	$$
	where $N_1:= \max\{ (2N_0)^{\gamma}, 2^n\}$.
	For this, taking $x_1 \in B_1$ such that 
	\begin{equation} \label{V(u-h)}
	\mathcal{M}_{B_3}(|D^2(u-h)|^{\gamma})(x_1) \le N_0^{\gamma},
	\end{equation}
	i.e. 
	\begin{equation*}
	\frac{1}{|B_\rho|} \int_{B_{\rho}(x_1)\cap B_3} |D^2(u-h)(y)|^{\gamma} dy  \leq N_0^{\gamma} \quad \text{for any  }  \rho>0 ,
	\end{equation*}
	we claim that 
	\begin{equation}\label{claim}
	\sup_{\tilde{\rho}>0}\, \dashint_{B_{\widetilde{\rho}}(x_1)} |D^2u(y)|^{\gamma} dy \leq N_1^{\gamma} 
	\end{equation}
	where $N_1=\max\{ (2N_0)^{\gamma}, 2^n\}$.  

	For $\widetilde{\rho}>2$, we see that $x_0 \in B_1 \subset B_{\widetilde{\rho}} (x_1) \subset B_{2\widetilde{\rho}}(x_0)$
	and then it follows that
	\begin{align*} 
	\dashint_{B_{\widetilde{\rho}}(x_1)} |D^2u(y)|^{\gamma} dy 
	&\leq  2^{n}\, \dashint_{B_{2{\widetilde{\rho}}}(x_0)} |D^2u(y)|^{\gamma} dy \leq 2^{n}
	\end{align*}
	by \eqref{con1}.
	Meanwhile for the case $\widetilde{\rho}\leq 2$, 
	we note that $B_{\widetilde{\rho}}(x_1) \subset B_3$. 
	Thus we get
	\begin{align*}
	&\sup_{\widetilde{\rho}\leq 2}\, \dashint_{B_{\widetilde{\rho}}(x_1)} |D^2u(y)|^{\gamma} dy\\ 
	& \qquad \leq 2^{\gamma-1}  \sup_{\widetilde{\rho}\leq 2}\, \left(  \dashint_{B_{\widetilde{\rho}}(x_1)} |D^2(u-h)(y)|^{\gamma} dy + \dashint_{B_{\widetilde{\rho}}(x_1)} |D^2h(y)|^{\gamma} dy \right) \\
	&\qquad \le 2^{\gamma-1} \left( \mathcal{M}_{B_3}(|D^2(u-h)|^{\gamma})(x_1)+ \sup_{ y\in B_3} |D^2h(y)|^{\gamma}\right)  \\
	&\qquad \le  (2N_0)^{\gamma} 
	\end{align*}
	by using \eqref{V(u-h)} and \eqref{L infty estimate_Du}.
	Hence we obtain \eqref{claim}.
	
	Therefore, we conclude 
	\begin{align*}
	|\{x\in B_1\,:\, \mathcal{M}(|D^2u|^{\gamma})(x)>N_1^{\gamma} \}|
	&< |\{ x\in B_1\,:\, \mathcal{M}_{B_3}(|D^2(u-h)|^\gamma)(x) > N_0^{\gamma}\}|\\
	&\leq \varepsilon |B_1|
	\end{align*}
	as desired. 
\end{proof}

The following directly comes from the previous lemma via the scaling argument. 

\begin{lem}\label{lemma2}
	For any $\varepsilon>0$, $0<r<1$ and $y\in B_1$
	there exists a small
	$\delta = \delta( n, \gamma, \mu, q, c_q, \varepsilon)>0$
	so that
	for any solution $u \in W_V^{2, \gamma}(\Omega)$ of 
\[
	-a_{ij} D_{ij} u+Vu=f \ \ \text{ in } \Omega \supset B_{6r}(y)
	\]
	 assuming that $\textbf{A}=(a_{ij})$ is $(\delta, 6r)$-vanishing and
	$V\in \mathcal{RH}_q$ with some $q>n/2$, if $r^{-2-n/\gamma}\|u\|_{L^{\gamma}(B_{6r}(y))} \le \delta$ and
	\begin{equation*}
	\{x\in \Omega\,:\, \mathcal{M}(|D^2u|^{\gamma})(x)\leq 1 \} \cap \{x\in \Omega\, :\, \mathcal{M}(|f|^{\gamma})(x)\leq \delta^{\gamma}\} \cap B_r(y) \neq \emptyset
	\end{equation*}
	for $1<\gamma\leq q$,  
	then we have
	\[
	|\{x\in \Omega\,:\, \mathcal{M}(|D^2u|^{\gamma})(x)>N_1^{\gamma} \}\cap B_r(y) | < \varepsilon |B_r(y)|
	\]
	where $N_1$ is given in Lemma~\ref{keylemma}.
\end{lem}

\begin{lem}\label{cor2}
Let $\omega \in A_{p}$ with $1< p<\infty$.
For any $\varepsilon_1>0$, $0<r<1$ and $y\in B_1$
there exists a small
$\delta = \delta( n, \gamma, \mu, q, c_q, p, [\omega]_p, \varepsilon_1)>0$
so that
for any solution $u \in W_V^{2, \gamma}(\Omega)$ of \eqref{second_order_ell} in $\Omega \supset B_6$
assuming that $\textbf{A}$ is $(\delta, 6)$-vanishing and $V\in \mathcal{RH}_q$ for some $q>n/2$, if $\|u\|_{L^{\gamma}(B_6)} \leq \delta$ and
	\begin{equation}\label{assump3}
	 w (\{x\in B_1\,:\, \mathcal{M}(|D^2u|^{\gamma})(x)>N_1^{\gamma} \}\cap B_r (y) ) \ge \varepsilon_1 w (B_r(y))
	\end{equation}
	then we have 
	\begin{equation}\label{result3}
	B_1 \cap B_r(y) \subset
	\{x\in B_1\,:\, \mathcal{M}(|D^2u|^{\gamma})(x)> 1 \} \cup \{x\in B_1\, :\, \mathcal{M}(|f|^{\gamma})(x)>\delta^{\gamma}\},
	\end{equation}
	where $N_1$ is given in Lemma~\ref{keylemma}.
\end{lem}

\begin{proof}
	Let us assume that \eqref{result3} is false under \eqref{assump3}. 
	Namely, for any $y\in B_1$ and $0<r<1$, there exists $x_0 \in B_1 \cap B_r(y) $ satisfying
	\begin{equation}\label{assump3-1} 
	\sup_{\rho>0}\frac{1}{|B_\rho|}\int_{B_{\rho}(x_0)\cap \Omega} |D^2u|^{\gamma}(x)\, dx \leq 1 
	\end{equation}
	and  
	\begin{equation}\label{assump3-2}
	\sup_{\rho>0}\frac{1}{|B_\rho|}\int_{B_{\rho}(x_0)\cap \Omega}  |f|^{\gamma}(x)\, dx \leq \delta^{\gamma} 
	\end{equation}
	where $\Omega \supset B_6$.

We will use the scaling argument. 
For fixed $y^{\ast} \in B_1$ and $0<R<1$, 
we let 
\[
\begin{aligned}
&\quad v(z)=R^2u\big(\frac{z-y^{\ast}}{R}\big), \quad g(z)=f\big(\frac{z-y^{\ast}}{R}\big),\\
&\widetilde{a}_{ij}(z)=a_{ij}\big(\frac{z-y^{\ast}}{R}\big)\quad \text{and} \quad \widetilde{V}(z)=\frac1{R^2} V\big(\frac{z-y^{\ast}}{R}\big) .
\end{aligned}
\] 
Then $v$ is a solution to
\begin{equation*}
	-\widetilde{a}_{ij}D_{ij}v+\widetilde{V}v=g \ \ \text{ in } \tilde{\Omega} \supset B_{6R}(y^{\ast})
\end{equation*}
where $\tilde{\Omega} =\{z= Rx+y^{\ast}\,:\, x\in \Omega\}$. 
Here, one can check that the coefficients $\widetilde{a}_{ij}$ and $\widetilde{V}$ satisfy the assumptions given as in Lemma \ref{lemma2}, thanks to \eqref{invariant}.
Also, we see that $$ R^{-2-\frac{n}{\gamma}}\| v \|_{L^\gamma(B_{6R}(y^{\ast}))} = \|u\|_{L^{\gamma}(B_{6})}\le \delta.$$ 
Moreover, 
we set 
$$
z_0 :=Rx_0+y^{\ast}.
$$
Then $z_0 \in B_{R}(y^{\ast})\cap B_{Rr}(Ry+y^{\ast}) \subseteq B_{R}(y^{\ast})$ for $x_0 \in B_1 \cap B_r(y) $
and it follows that
\begin{equation*}
	\sup_{\rho>0}\frac{1}{|B_\rho|}\int_{B_{\rho}(z_0)\cap \widetilde{\Omega}} |D^2v|^{\gamma}(z) dz 
	=	\sup_{\rho>0}\frac{1}{|B_{\frac{\rho}{R}}|}\int_{B_{\frac{\rho}{R}}(x_0)\cap \Omega} |D^2u|^{\gamma}(x) dx \leq 1 
\end{equation*}
and 
\begin{equation*}
	\sup_{\rho>0}\frac{1}{|B_\rho|}\int_{B_{\rho}(z_0)\cap \tilde{\Omega}}  |f|^{\gamma}(z) dz \leq \delta^{\gamma}
\end{equation*}
from \eqref{assump3-1} and \eqref{assump3-2} respectively.
All assumptions given in Lemma \ref{lemma2} are satisfied and thus
we use Lemma \ref{lemma2}
 to obtain
\[
\begin{aligned}
|\{z\in \widetilde{\Omega}\,:\, \mathcal{M}(|D^2v|^{\gamma})(z)>N_1^{\gamma} \}\cap B_R(y^{\ast}) | < \varepsilon_0 |B_R(y^{\ast})|
\end{aligned}
\]
for any $\varepsilon_0>0$.
Then we derive
\[
\begin{aligned}
&\frac{1}{|B_{r}(y)|} |\{x\in B_1\,:\, \mathcal{M}(|D^2u|^{\gamma})(x)>N_1^{\gamma} \}\cap B_{r}(y) |
\\
& \qquad = \frac{1}{|B_{r}(y)|} \int_{\{x\in B_1\cap B_r(y)\,:\, \mathcal{M}(|D^2u|^{\gamma})(x)>N_1^{\gamma} \} }1 \,dx\\
& \qquad\le \frac{1}{|B_{r}(y)|} \int_{\{x\in \Omega \cap B_1 \,:\, \mathcal{M}(|D^2u|^{\gamma})(x)>N_1^{\gamma} \}} 1 \,dx\\
& \qquad = \frac{1}{|B_{r}(y)|} \int_{\{z\in \widetilde{\Omega} \cap B_R(y^{\ast}) \,:\, \mathcal{M}(|D^2v|^{\gamma})(z)>N_1^{\gamma} \} } \frac1{R^n} \,dz\\
& \qquad = \frac1{r^n} \frac{1}{|B_{R}(y^{\ast})|} \int_{\{z\in \widetilde{\Omega} \cap B_R(y^{\ast}) \,:\, \mathcal{M}(|D^2v|^{\gamma})(z)>N_1^{\gamma} \} }  \,dz\\
& \qquad = \frac1{r^n} \frac{1}{|B_{R}(y^{\ast})|} |\{z\in \widetilde{\Omega} \cap B_R(y^{\ast}) \,:\, \mathcal{M}(|D^2v|^{\gamma})(z)>N_1^{\gamma} \} | < \frac{1}{r^n}\varepsilon_0.
\end{aligned}
\]
	By applying Lemma \ref{property_A} to this, we get
	\begin{align*}
	&\frac{\omega(\{x\in \Omega\,:\, \mathcal{M}(|D^2u|^{\gamma})(x)>N_1^{\gamma} \}\cap B_r(y))}{\omega(B_r(y))} \\
	&\qquad \leq c \left( \frac{|\{x\in \Omega\,:\, \mathcal{M}(|D^2u|^{\gamma})(x)>N_1^{\gamma} \}\cap B_r(y)|}{|B_r(y)|} \right)^{\sigma} \\
	&\qquad <\frac{c\varepsilon_0^{\sigma}}{r^{n\sigma}} < \varepsilon_1
	\end{align*}
where $ \sigma \in (0,1)$ is given in Lemma~\ref{property_A}, by taking $\varepsilon_0$ so that the last inequality holds.
But it is a contradiction to \eqref{assump3}.
Hence the proof is done.
\end{proof}

\section{Proof of main result}\label{se5}

\subsection{Covering Lemma}
The modified covering lemma in the weighted setting is shown in \cite[Lemma~3.3]{BL15} for a parabolic case. The proof of the elliptic case is quite similar to that of the parabolic one, but we give it for the completeness of this paper.
\begin{lem}\label{covering}
Let $w \in A_{p}$, $1<p<\infty$ and let $0<\varepsilon <1$.
Suppose $E$ and $F$ are measurable sets with $E \subset F \subset B_1$ such that
\begin{itemize}
	\item[(1)] $w(E)<\varepsilon w(B_1)$
	\item[(2)] for any $y\in B_1$ and $r\in (0,1]$,
	$$
	w(E \cap B_r(y)) \ge \varepsilon w(B_r(y)) \quad \text{implies} \quad B_r(y)\cap B_1 \subset F.
	$$
\end{itemize} 
Then $w(E) \leq c\varepsilon [w]^2_{p} 10^{np}  w(F)$. 
\end{lem}

\begin{proof}
Since $w(E) <\varepsilon w(B_1)$, there exists a small $r(x)>0$ for almost all $x\in E$ such that
$$
w(E \cap B_{r(x)}(x)) =\varepsilon w(B_{r(x)}(x)) $$
and 
\begin{equation*}
w(E \cap B_{r}(x)) <\varepsilon w(B_{r}(x)) \qquad \text{for } r>r(x).
\end{equation*}

By using the Vitali covering lemma, 
there is $\{ x_k\}_{k\in \mathbb{N}} \subseteq E$ so that $B_{r_k}(x_k)$ are mutually disjoint and 
$$E \subset \bigcup_{k\in \mathbb{N}} B_{5r_k}(x_k) \cap B_1$$
with $r_k=r(x_k)$ given in the above.
Then we see that
\begin{align}\label{mes1}
w(E \cap B_{5r_k}(x_k)) 
&< \varepsilon w( B_{5r_k}(x_k)) \nonumber\\
&< \varepsilon [w]_{p} 5^{np} w( B_{r_k}(x_k) ) 
\end{align}
by \eqref{measure comparison}.
Also, along with the fact 
$$
\sup_{0<r<1} \sup_{x\in B_1} \frac{|B_r(x)|}{|B_r(x) \cap B_1|} \leq 2^n
$$
we have 
\begin{equation}\label{mes2}
w(B_{r_k}(x_k))  \leq [w]_p 2^{np} w(B_{r_k}(x_k) \cap B_1).
\end{equation} 

Hence by combining \eqref{mes1} and \eqref{mes2} we conclude that
\begin{align*}
w(E) = w\big(E \cap \bigcup_k B_{5{r_k}}(x_k)\big) 
& \leq \sum_k  w(E \cap B_{5{r_k}}(x_k))\\
& \leq  \varepsilon [w]_{p} 5^{np} \sum_k  w( B_{r_k}(x_k) )\\
& \leq  c\varepsilon [w]^2_{p} 5^{np} 2^{np} \sum_k  w( B_{r_k}(x_k) \cap B_1 )\\
& \leq  c\varepsilon [w]^2_{p} 10^{np} w\big( \bigcup_k B_{r_k}(x_k) \cap B_1 \big)\\
& \leq  c\varepsilon [w]^2_{p} 10^{np} w( F )
\end{align*} 
from the condition $(2)$.
\end{proof}

\begin{lem}\label{iteration}
Under the same assumptions in Lemma \ref{cor2} with $0<\varepsilon_1<1$, if
\begin{equation}\label{con3}
w(\{ x\in B_1 : \mathcal{M}(|D^2u|^{\gamma})(x)>N_1^{\gamma} \}) <\varepsilon_1 w(B_1),
\end{equation}
then we have 
\begin{align}\label{decay}
&w(\{x\in B_1\,:\, \mathcal{M}(|D^2u|^{\gamma})(x)>(N_1^\gamma)^k \}) \nonumber\\
&\qquad \leq \sum_{i=1}^k \varepsilon_2^i w( \{x\in B_1\,:\, \mathcal{M}(|f|^{\gamma})(x)>\delta^{\gamma}(N_1^{\gamma})^{k-i} \} )\nonumber\\
&\qquad \quad +\varepsilon_2^k w(\{x\in B_1\,:\, \mathcal{M}(|D^2u|^{\gamma})(x)>1 \})
\end{align}
for $k\in \mathbb{N}$ where $\varepsilon_2:=  [w]^2_{p} 10^{np} \varepsilon_1$ and $N_1$ is given in Lemma~\ref{keylemma}. 
\end{lem}

\begin{proof}
	We shall apply Lemma \ref{covering} with  
	$$
	E=\{x\in B_1\,:\, \mathcal{M}(|D^2u|^{\gamma})(x)>N_1^\gamma \}
	$$
	and 
	$$
	F =\{x\in B_1\,:\, \mathcal{M}(|D^2u|^{\gamma})(x)> 1 \} \cup \{x\in B_1\, :\, \mathcal{M}(|f|^{\gamma})(x)> \delta^{\gamma}\}.
	$$
	
	It is easy to check that $E \subset F \subset B_1$ since $N_1 >1$.
	The first hypothesis stated in Lemma \ref{covering} is satisfied by \eqref{con3},
	and the second hypothesis is satisfied from Lemma \ref{cor2}.
	Thus by applying Lemma \ref{covering}, 
	we obtain 
	\begin{align*}
	&w\left(\{x\in B_1\,:\, \mathcal{M}(|D^2u|^{\gamma})(x)>N_1^\gamma \} \right) \\
	&\qquad \leq \varepsilon_2 w(\{x\in B_1\,:\, \mathcal{M}(|D^2u|^{\gamma})(x)> 1 \})\\
	&\qquad \quad + \varepsilon_2 w( \{x\in B_1\, :\, \mathcal{M}(|f|^{\gamma})(x)> \delta^{\gamma}\})
	\end{align*}
	with $\varepsilon_2 :=[w]^2_{p} 10^{np} \varepsilon_1$.
	Hence we get \eqref{decay} when $k=1$.
	
	For the case $k\ge2$, 
	we note 
	$$
	-a_{ij}D_{ij}(N_1^{-(k-1)}u) +V(N_1^{-(k-1)}u) =N_1^{-(k-1)}f
	$$
	from which the similar process as above works with $u$ and $f$ replaced by $N_1^{-(k-1)}u$ and $N_1^{-(k-1)}f $ respectively.
	Thus it follows that
	\begin{align}\label{decayk}
	&w(\{x\in B_1\,:\, \mathcal{M}(|D^2u|^{\gamma})(x)>(N_1^\gamma)^k \})\nonumber \\
	&\qquad \leq\, \varepsilon_2 w(\{x\in B_1\,:\, \mathcal{M}(|f|^{\gamma})(x)> (N_1^{\gamma})^{k-1}\delta^{\gamma} \}) \nonumber\\
	&\qquad \quad +\varepsilon_2 w(\{x\in B_1\,:\, \mathcal{M}(|D^2u|^{\gamma})(x)>(N_1^{\gamma})^{k-1}\})
	\end{align}
	from the fact that
	$$
	\{x\in B_1\,:\, \mathcal{M}(|D^2(N_1^{-(k-1)}u)|^{\gamma})(x)> N_1^{\gamma} \} =
	\{x\in B_1\,:\, \mathcal{M}(|D^2u|^{\gamma})(x)>N_1^{k \gamma}  \} 
	$$
	and, similarly,
	$$
	\{x\in B_1\,:\, \mathcal{M}(|N_1^{-(k-1)}f|^{\gamma})(x)>\delta^{\gamma} \}
	=\{x\in B_1\,:\, \mathcal{M}(|f|^{\gamma})(x)>(N_1^{k-1}\delta)^{\gamma} \}.
	$$
	We deal with the last term in \eqref{decayk}
	by using the same procedure with $u$ and $f$ replaced by $N_1^{k-2}u$ and $N_1^{k-2}f$ 
	so that \eqref{decayk} becomes
	\begin{align*}
	&w(\{x\in B_1\,:\, \mathcal{M}(|D^2u|^{\gamma})(x)>(N_1^\gamma)^k \}) \nonumber \\
	&\qquad \leq \sum_{i=1,2} \varepsilon_2^i w(\{x\in B_1\,:\, \mathcal{M}(|f|^{\gamma})(x)>\delta^{\gamma}(N_1^{\gamma})^{k-i} \})\nonumber\\
	&\qquad \quad +\varepsilon_2^2 w(\{x\in B_1\,:\, \mathcal{M}(|D^2u|^{\gamma})(x)>(N_1^{\gamma})^{k-2} \}).
	\end{align*}
	In this way, we repeat the process $k-2$ times.
	Hence we obtain \eqref{decay} as desired.
\end{proof}

\subsection{Proof of Theorem~\ref{mainresult}}

Now we are ready to prove our main theorem.
Let $1<\gamma \leq  \min\{ \tilde{p}, q\}$ where $\tilde{p}$ is given in Lemma~\ref{Phi-integrability_f}.

For any solution $u$ of \eqref{second_order_ell}
we first see that 
\begin{align*}
\| Vu\|_{W_w^{2,\Phi}(B_{1})} 
&=\|f+a_{ij}D_{ij}u\|_{L_w^\Phi(B_{1})}\nonumber\\
&\leq \|f\|_{L_w^\Phi(B_{1})} +\|a_{ij}\|_{L^{\infty}(B_1)} \|D^2u\|_{L_w^\Phi(B_{1})} \nonumber\\
&\leq \|f\|_{L_w^\Phi(B_1)} + c\|D^2u\|_{L_w^\Phi(B_{1})}
\end{align*}
since $a_{ij}\in L^{\infty}$ from the symmetry and uniform ellipticity conditions.
Thus it is sufficient to show that 
\begin{equation}\label{mainclaim_a}
\|D^2u\|_{L_w^\Phi(B_{1})}\leq c(\|f\|_{L_w^\Phi(B_6)} +\|u\|_{L^{\gamma}(B_6)}).
\end{equation}
Without loss of generality, let us assume that
\begin{equation}\label{as1}
\|f\|_{L_w^{\Phi}(B_6)} \leq \delta \quad \textrm{and} \quad \|u\|_{L^{\gamma}(B_6)} \leq \delta
\end{equation}
for small $\delta>0$ chosen later.
Indeed, we consider
$$
\tilde{u}:=\frac{\delta u }{ \|f\|_{L_w^\Phi(B_6)} +\|u\|_{L^\gamma(B_6)}}\quad \text{and} \quad \tilde{f}= \frac{\delta f}{\|f\|_{L_w^\Phi(B_6)} +\|u\|_{L^\gamma(B_6)}},
$$
instead of $u$ and $f$ respectively. It is clear that $\tilde{u}$ solves the equation \eqref{second_order_ell} with $\tilde{f}$ instead of $f$.
Letting 
$
M:= \|f\|_{L_w^\Phi(B_6)} +\|u\|_{L^\gamma(B_6)},
$
we have 
\[
\|\tilde{f}\|_{L_w^{\Phi}(B_6)}  = \frac{\delta}{M}  \| f\|_{L_w^{\Phi}(B_6)} \le \delta
\]
and
\[\|\tilde u\|_{L^{\gamma}(B_6)}  =\frac{\delta}{M}  \|u\|_{L^\gamma(B_6)} \le  \delta
\]
from which the assumptions in \eqref{as1} hold true. 
Thus \eqref{mainclaim_a} turns to showing
\begin{equation}\label{main_claim}
\|D^2u\|_{L_w^{\Phi}(B_{1})}\leq c
\end{equation}
under \eqref{as1}.

From Lemma~\ref{Phi-integrability_f}, we have
\[
 \|f\|_{L^{\gamma}(B_6)}  \le c \|f\|_{L^{\tilde{p}}(B_6)}  \le c \|f\|_{L_w^\Phi (B_6)}  \le c\,\delta
\]
for some constant $c = c(n,q, i(\Phi), [w]_{i(\Phi)})>0$.
Then by using \eqref{measure comparison}, \eqref{weak1-1} and applying \cite[Theorem~13]{BBHV},
it follows that
\begin{align*}
&\frac{w( \{ x\in B_1\,:\, \mathcal{M}(|D^2u|^{\gamma})(x)>N_1^{\gamma} \} ) }{w(B_1)}\nonumber\\
&\leq c\left( \frac{ |\{ x\in B_1\,:\, \mathcal{M}(|D^2u|^{\gamma})(x)>N_1^{\gamma} \}|}{|B_1|}\right)^{\sigma} \leq c \left( \frac{ \|D^2u\|_{L^{\gamma}(B_1)}^{\gamma} }{N_1^{\gamma}|B_1|}\right)^{\sigma}\nonumber\\
&\leq c \left( \frac{ \|f\|_{L^{\gamma}(B_6)}^{\gamma}+ \|u\|_{L^{\gamma}(B_6)}^{\gamma} }{N_1^{\gamma}|B_1|}\right)^{\sigma} 
\le c \left(\frac{\delta^{\gamma}}{{N_1^{\gamma}|B_1|}}\right)^{\sigma} 
\end{align*}
under \eqref{as1}.
Here, a constant $c$ depends on $n, \mu, q, c_q, i(\Phi), [w]_{i(\Phi)}$, and $N_1$ is given in Lemma~\ref{keylemma}.
For any $\varepsilon_1>0,$
we now take a small $\delta = \delta(n, \mu, q, c_q, i(\Phi), [w]_{i(\Phi)},\varepsilon_1)>0$ so that 
\[
c \left(\frac{\delta^{\gamma}}{{N_1^{\gamma}|B_1|}}\right)^{\sigma}   < \varepsilon_1.
\]
Thus we derive  
\begin{equation}\label{as2}
w(\{ x\in B_1\,:\, \mathcal{M}(|D^2u|^{\gamma})(x)>N_1^{\gamma} \}) < \varepsilon_1 w(B_1).
\end{equation}

From Lemmas \ref{measure theory} and \ref{Mbdd in wO}, it follows that
\begin{align}\label{m1}
\int_{B_1} \Phi(|D^2u|) w\, dx  &= \int_{B_1} \Phi((|D^2u|^{\gamma})^{\frac{1}{\gamma}}) w \, dx \nonumber\\
& \le  \int_{B_1} \Phi(\mathcal{M}(|D^2u|^{\gamma})^{\frac1\gamma}) w \, dx\le c \left( w(B_1)+ S_0\right) 
\end{align}
for some constant $c = c(n, \mu, q, c_q, w, \Phi, i(\Phi))>0$,
where 
 \begin{align*}
S_0 &:=\sum_{k=1}^{\infty} \Phi(N_1^
 k) \omega( \{ x\in B_1\,:\, \mathcal{M}(|D^2u|^{\gamma})^{\frac1\gamma}(x)>N_1^ k\})\\
 & = \sum_{k=1}^{\infty} \Phi(N_1^
 k) \omega( \{ x\in B_1\,:\, \mathcal{M}(|D^2u|^{\gamma})(x)>N_1^{\gamma k}\}).
\end{align*} 
Thanks to \eqref{as2}, we use Lemma \ref{iteration} to yield that 
 \begin{align}\label{m2}
S_0 \leq 
&\sum_{k=1}^{\infty}\Phi(N_1^{
 k})\sum_{i=1}^k \varepsilon_2^i w(\{x\in B_1\,:\, \mathcal{M}(|f|^{\gamma})(x)>\delta^{\gamma}(N_1^{\gamma})^{k-i} \}) \nonumber \\
&\quad +\sum_{k=1}^{\infty}\Phi(N_1^{
 k}) \varepsilon_2^k w(\{x\in B_1\,:\, \mathcal{M}(|D^2u|^{\gamma})(x)>1 \})
\end{align} 
for some $N_1>1$ and $\varepsilon_2=[w]_{i(\Phi)}^2 10^{ni(\Phi)} \varepsilon_1$.
Note that $\Phi(N_1) \le \tau \Phi(1)$ for some constant $\tau=\tau(N_1)>1$ since $\Phi \in \Delta_2 \cap \nabla_2 $. 
We iterate this inequality to have $\Phi(N_1^{k}) \le \tau^i\Phi(N_1^{k-i})$ for $1\leq i\leq k$.
The last term in \eqref{m2} is then estimated as 
\begin{equation}\label{m3}
\sum_{k=1}^{\infty}\Phi(N_1^{
 k}) \varepsilon_2^k w(\{x\in B_1\,:\, \mathcal{M}(|D^2u|^{\gamma})(x)>1 \}) 
\leq \sum_{k=1}^{\infty} (\tau \varepsilon_2 )^{k}\Phi(1)  w(B_1) \leq c
\end{equation}
by taking $\epsilon_1 \in (0, 1)$ to be sufficiently small $\tau \varepsilon_2 < 1 $. 
Here, a constant $c>0$ depends on $n, w, \Phi, i(\Phi)$ and $[w]_{i(\Phi)}$.

Meanwhile, the first term in the right hand side of \eqref{m2} becomes 
\begin{align}\label{m4}
&\sum_{k=1}^{\infty}\Phi(N_1^{k}) 
\sum_{i=1}^k \varepsilon_2^i w(\{x\in B_1\,:\, \mathcal{M}(|f|^{\gamma})(x)>\delta^{\gamma}(N_1^{\gamma})^{k-i} \}) \nonumber\\
&\qquad = \sum_{i=1}^{\infty} (\tau \varepsilon_2)^i  \sum_{k=i}^{\infty} \Phi(N_1^{k-i}) w(\{x\in B_1\,:\, \mathcal{M}(|f|^{\gamma})(x)>\delta^{\gamma}(N_1^{\gamma})^{k-i} \}).
\end{align}
Letting $j=k-i$, we then derive that  
\begin{align}\label{m5}
&  \sum_{k=i}^{\infty} \Phi(N_1^{k-i}) w(\{x\in B_1\,:\, \mathcal{M}(|\delta^{-1}f|^{\gamma})(x)>(N_1^{\gamma})^{k-i} \})\nonumber \\ 
&\qquad = \Phi(1) w(\{x\in B_1\,:\, \mathcal{M} (|f|^{\gamma})(x)>\delta^{\gamma} \}) \nonumber \\
&\qquad \quad + \sum_{j=1}^{\infty}\Phi(N_1^{j}) w(\{x\in B_1\,:\, \mathcal{M} (|\delta^{-1}f|^{\gamma})^{\frac{1}{\gamma}}(x)>N_1^{j} \}) \nonumber \\
& \qquad \leq \Phi(1)w( B_1)  + c \int_{B_1} \Phi( \mathcal{M} (|\delta^{-1}f|^{\gamma})^{\frac1\gamma})w\,dx
\end{align}
from \eqref{estimate_S}.
Here, if we let $\Phi_{\gamma}(\rho) = \Phi(\rho^{\frac{1}{\gamma}})$, then $\Phi_{\gamma}$ is a $N$-function and  satisfies the $\Delta_2 \cap \nabla_2$-condition.   
Note that 
$$
i(\Phi_{\gamma}) = \frac{i(\Phi) }{\gamma } >\frac{ i(\Phi)-\tilde{\varepsilon}/2 }{ \min \{ q, \tilde{p}\}  } \ge\frac{ i(\Phi)-\tilde{\varepsilon}/2 }{ \tilde{p}  } =i(\Phi)-\tilde{\varepsilon}
$$ 
from the definition of $\tilde p$ in the proof of Lemma~\ref{Phi-integrability_f}.
Then the assumption that $w \in A_{i(\Phi)}$ implies that $w \in A_{i(\Phi)-\tilde{\varepsilon}} \subset A_{i(\Phi_{\gamma})}.$
Therefore
applying Lemma~\ref{Mbdd in wO} with $\Phi_{\gamma}$, we obtain that 
\begin{align}\label{m6}
& \int_{B_1} \Phi( \mathcal{M} (|\delta^{-1}f|^{\gamma})^{\frac1\gamma})w\,dx
= \int_{B_1} \Phi_{\gamma}( \mathcal{M} (|\delta^{-1}f|^{\gamma}))w\,dx \nonumber \\
&\qquad\le c \int_{B_1} \Phi_{\gamma}( |\delta^{-1}f|^{\gamma})w\,dx = c \int_{B_1} \Phi( |\delta^{-1}f|)w\,dx \le c
\end{align}
where we used \eqref{as1} together with \eqref{unitballproperty} in the last inequality.
Hence we insert the resulting estimates \eqref{m5}-\eqref{m6}  into \eqref{m4} in order to get
\begin{equation}\label{result_m4}
\sum_{k=1}^{\infty}\Phi(N_1^{k}) 
\sum_{i=1}^k \varepsilon_2^i w(\{x\in B_1\,:\, \mathcal{M}(|f|^{\gamma})(x)>\delta^{\gamma}(N_1^{\gamma})^{k-i} \}) \le  c \sum_{i=1}^{\infty} (\tau \varepsilon_2)^i  \le c 
\end{equation}
for some constant $c = c(n, \mu, q, c_q, w, \Phi, i(\Phi), [w]_{i(\Phi)})>0$.
Consequently, combining \eqref{m1}-\eqref{m3} and \eqref{result_m4},
we obtain the desired estimates \eqref{main_claim}.

\end{document}